\documentclass[aos,preprint]{imsart}

\RequirePackage[OT1]{fontenc}
\RequirePackage{amsthm,amsmath}
\RequirePackage[numbers]{natbib}
\RequirePackage[colorlinks,citecolor=blue,urlcolor=blue]{hyperref}
\usepackage{amsmath,graphicx,amssymb,bm}
\usepackage{url}

\arxiv{arXiv:0000.0000}

\startlocaldefs
\numberwithin{equation}{section}
\theoremstyle{plain}
\newtheorem{thm}{Theorem}[section]
\endlocaldefs

\graphicspath{{fig/}}

\newcommand{\ba}{{\mathbf a}}
\newcommand{\bb}{{\mathbf b}}
\newcommand{\bc}{{\mathbf c}}
\newcommand{\bd}{{\mathbf d}}
\newcommand{\be}{{\mathbf e}}
\newcommand{\bbf}{{\mathbf f}}

\newcommand{\bx}{{\mathbf x}}

\newcommand{\by}{{\mathbf y}}
\newcommand{\bz}{{\mathbf z}}

\newcommand{\bH}{{\mathbf H}}
\newcommand{\bR}{{\mathbf R}}
\newcommand{\bP}{{\mathbf P}}
\newcommand{\bI}{{\mathbf I}}

\newcommand{\bzeta}{{\bm \zeta}}
\newcommand{\bxi}{{\bm \xi}}
\newcommand{\btheta}{{\bm \theta}}
\newcommand{\bZ}{{\bm 0}}

\newcommand{\Hc}{{\mathcal H}}
\newcommand{\Lc}{{\mathcal L}}
\newcommand{\Mc}{{\mathcal M}}
\newcommand{\Jc}{{\mathcal J}}
\newcommand{\Ic}{{\mathcal I}}

\newcommand{\rank}{\textrm{rank}}
\newcommand{\BesselI}{I}

\begin{document}

\begin{frontmatter}
\title{Sensor fusion for bimodal generalized likelihood ratio test with unknown noise variances}
\runtitle{Bimodal fused GLRT with unknown variances}

\begin{aug}
\author{\fnms{Borislav N.} \snm{Oreshkin}\thanksref{m1}\ead[label=e1]{boris.oreshkin@mail.mcgill.ca}},
\author{\fnms{Ekaterina Y.} \snm{Turkina}\thanksref{m2}\ead[label=e2]{ekaterina.turkina@hec.ca}}

\runauthor{Oreshkin and Turkina}

\affiliation{Universit\'{e} de Montr\'{e}al\thanksmark{m1} and HEC Montr\'{e}al\thanksmark{m2}}

\address{D\'{e}partement d'Informatique\\
et de Recherche Op\'{e}rationnelle\\
Universit\'{e} de Montr\'{e}al\\
Pavillon Andr\'{e}-Aisenstadt\\
CP 6128 succ Centre-Ville\\
Montr\'{e}al, Qu\'{e}bec\\
H3C 3J7, CANADA \\
\printead{e1}
\phantom{E-mail:\ }}

\address{Department of International Business\\
HEC Montr\'{e}al \\
3000, chemin de la C\^{o}te-Sainte-Catherine\\
Montr\'{e}al,  Qu\'{e}bec\\
H3T 2A7, CANADA \\
\printead{e2}
\phantom{E-mail:\ }}

\end{aug}

\begin{abstract}
In this paper we address the problem of sensor fusion. We formulate the joint detection problem using a general linear observation model and inter-modality independence assumption for noises. We derive the fusion architecture based on the generalized likelihood ratio principle and calculate the expressions for the distributions of the test statistic under the signal present and the null hypotheses. To obtain these results we develop a methodology for the joint detection algorithm analysis based on the theory of the Meijer G-function.
\end{abstract}



\end{frontmatter}

\section{Introduction}
\label{sec:Introduction}

 This paper focuses on joint detection with parameter uncertainty. Joint detection involves fusion of data from several sensors (measurement modalities) and is often necessary, because a single sensor has too low detection probability (high false alarm rate). Joint detection has wide range of applications. For example, it is used in landmine detection~\cite{Cremer2001}, multimodality breast cancer detection~\cite{Kir2011} and multisite radar~\cite{Chernyak1998}. In this paper we significantly generalize and extend the statistical analysis developed by Kirshin et al.~\cite{Kir2011}, where a joint breast caner detection system using two sensor modalities (ultrawide-band radar and microwave-induced thermoacoustics) was presented. The main focus of the paper was on numerical experiments showing the potential of the joint detection system. Our current paper focuses on developing a general data-level fusion rule based on the
generalized maximum likelihood (GLR) approach and on thoroughly analyzing the distributions of the resulting test statistic. The contribution of our paper is thus (i) the development of a new class of GLR based probabilistic fusion rules, (ii) theoretical analysis of their detection performance, (iii) methodology for the analysis of the fusion rules based on the theory of the Meijer G-function. Although our study was motivated by the concrete application described in~\cite{Kir2011}, we believe that the results presented in this paper have much more general applicability. They can be used for the statistical analysis and design of a wide range of sensor fusion systems that can be described by the general signal model presented in section~\ref{sec:Problem_Statement}.

The rest of the paper is organized as follows. Section~\ref{sec:Problem_Statement} formally defines the signal models and the problem to
be solved. Section~\ref{sec:GLRT_based_fusion_rule} describes the GLR based fusion rule and Section~\ref{sec:Distribution_of_the_test_statistic} analyzes the distributions of the GLR based fused test statistic. Section~\ref{sec:Discussion_of_Results} provides discussion of our results and Section~\ref{section::Concluding Remarks} concludes the paper.

\section{Problem Statement}
\label{sec:Problem_Statement}

In this paper we consider the classical linear observation model
resulting in the following quasi-deterministic signal description
under signal present hypothesis $\Hc_1$:
\begin{align}
\bx &= \bH_{\bx} \btheta_{\bx} + \bxi, \label{E:r_sig_mod_1} \\
\by &= \bH_{\by} \btheta_{\by} + \bzeta. \label{E:r_sig_mod_2}
\end{align}
Here $\bx = [\bx_{1}, \ldots, \bx_{N}]^T$ and $\by = [\by_{1},
\ldots, \by_{M}]^T$ are waveforms observed by two different sensors
when hypothesis $\Hc_1$ (signal present) holds true. These waveforms
consist of the signal contributions given by the observation
matrices $\bH_{\bx}$ and $\bH_{\by}$ and two sets of unknown
deterministic parameters $\btheta_{\bx}$ and $\btheta_{\by}$; and
interfering Gaussian noises $\bxi$ and $\bzeta$ with zero mean and
covariance matrices $\sigma^2_{\bxi}\bR_{\bxi}$ and
$\sigma^2_{\bzeta}\bR_{\bzeta}$. Note that the adopted general
classical linear observation model contains many important detection
problems as special cases. For example, signals with unknown
amplitude and/or phase, signal with unknown arrival time and/or
frequency, signals received by an antenna array can all be
represented using this model via proper choice of observation matrix
$\bH$ and parametrization $\btheta$.

In this paper we assume that noises $\bxi$ and $\bzeta$ are independent and that the $\Hc_0$ hypothesis corresponds to the noise only observation scenario: $\bx = \bxi$, $\by = \bzeta$. The independency assumption can be justified in many practical situations. For example, when physics that govern measurement process are significantly different for the two sensors or measurements are significantly separated in space, time, or
frequency domains, this assumption holds. In fact, from the system
design perspective that would be the best sensor configuration, when
sensor fusion has potential to provide significant information gain.
On the contrary, little fusion gain is to be expected when sensor
noises are strongly correlated. We also assume that the noise
covariance matrices are known up to the scaling factors
$\sigma^2_{\bxi}$ and $\sigma^2_{\bzeta}$ and we treat these as the
nuisance parameters.

The goal of this paper is to derive the fusion rule for the
Generalized Likelihood Ratio Test (GLRT) based detector and to obtain
the exact non-asymptotic expressions for the test statistic
probability density functions (PDFs) under both $\Hc_0$ and $\Hc_1$.

\section{GLRT based fusion rule}
\label{sec:GLRT_based_fusion_rule}

The GLRT performs the comparison of the GLR $L_G(\bx,\by)$ against
the threshold $\gamma$:
\begin{align} \label{eq:Tst}
L_G(\bx,\by)
\overset{\mathcal{H}_{1}}{\underset{\mathcal{H}_{0}}{\gtrless}} ~
\gamma.
\end{align}
The GLR is obtained by plugging the maximum likelihood estimates
(MLEs) of unknown parameters under each hypothesis into the
likelihood ratio~\cite{Kay1998}. Under the assumptions stipulated in
Section~\ref{sec:Problem_Statement}, $L_G(\bx,\by)$ can be
factorized as follows:
\begin{align}
L_G(\bx,\by) = \frac{
p(\bx|\hat{\theta}_{\bx},\hat{\sigma}^2_{\bxi,1};\Hc_1)}{p(\bx|\hat{\sigma}^2_{\bxi,0};\Hc_0)}
\frac{p(\by|\hat{\theta}_{\by},\hat{\sigma}^2_{\bzeta,1};\Hc_1) }{
p(\by|\hat{\sigma}^2_{\bzeta,0};\Hc_0) }~. \label{eq:L_G_fact}
\end{align}
MLEs of the unknown parameters are presented in Appendix~\ref{app:MLEs}. Substituting them into~\eqref{eq:L_G_fact} results in:
\begin{align} \label{eq:L_G_final}
  \displaystyle
    L_G(\bx,\by) &= \left( \frac{\bx^T \bR_{\bxi}^{-1} \bx}{\bx^T \bR_{\bxi}^{-1/2} \mathbf{P}_{\bx}^{\bot} \bR_{\bxi}^{-1/2} \bx}
    \right)^{\frac{N}{2}}
    \left( \frac{\by^T \bR_{\bzeta}^{-1} \by}{\by^T \bR_{\bzeta}^{-1/2} \mathbf{P}_{\by}^{\bot} \bR_{\bzeta}^{-1/2} \by} \right)^{\frac{M}{2}}
\end{align}
The test statistic $L_G(\bx,\by)$ can thus be represented as the
product of two exponentiated random variables, $Z =
Z_{\bx}^{\frac{N}{2}} Z_{\by}^{\frac{M}{2}}$, of the form
\begin{align}
Z_{\bx} = \frac{\bx^T \bR_{\bxi}^{-1} \bx}{\bx^T \bR_{\bxi}^{-1/2}
\mathbf{P}_{\bx}^{\bot} \bR_{\bxi}^{-1/2} \bx} \text{ and } Z_{\by}
= \frac{\by^T \bR_{\bzeta}^{-1} \by}{\by^T \bR_{\bzeta}^{-1/2}
\mathbf{P}_{\by}^{\bot} \bR_{\bzeta}^{-1/2} \by}
\end{align}
Using the properties of signal projection matrices $\mathbf{P}_{\bx}$ and $\mathbf{P}_{\by}$ outlined
in Appendix~\ref{app:MLEs}, random variables $Z_{\bx}$ and $Z_{\by}$ can be further
represented as the following configuration of independent random
variables:
\begin{align} \label{eqn:individual_test_statistics}
Z_{\bx} = \frac{S_{\bx} + R_{\bx}}{S_{\bx}}, Z_{\by} = \frac{S_{\by}
+ R_{\by}}{S_{\by}}.
\end{align}
Here $S_{\bx} = \bx^T \bR_{\bxi}^{-1/2} \mathbf{P}_{\bx}^{\bot}
\bR_{\bxi}^{-1/2} \bx$, $R_{\bx} = \bx^T \bR_{\bxi}^{-1/2}
\mathbf{P}_{\bx} \bR_{\bxi}^{-1/2} \bx$ and $S_{\by} = \by^T
\bR_{\bzeta}^{-1/2} \mathbf{P}_{\by}^{\bot} \bR_{\bzeta}^{-1/2}
\by$, $R_{\by} = \by^T \bR_{\bzeta}^{-1/2} \mathbf{P}_{\by}
\bR_{\bzeta}^{-1/2} \by$.

It is interesting to note that $Z_{\bx}$ and $Z_{\by}$ are the GLRT
test statistics for the individual samples $\bx$ and $\by$
respectively. Since transformation $(\cdot)^{1/(N/2 + M/2)}$ is a
monotonically increasing function, it is not hard to see that the
test statistic $L_G(\bx,\by)^{1/(N/2 + M/2)}$ is equivalent
to~\eqref{eq:L_G_final}. The GLRT based joint processing thus leads
to the weighted geometric mean based fusion architecture. This
statement can be straightforwardly generalized to the multi-sensor
setting with samples $\bx, \by, \bz, \ldots$

\section{Distributions of the fused test statistic}
\label{sec:Distribution_of_the_test_statistic}

In this section we derive the distributions of the test statistic
under hypotheses $\Hc_0$ and $\Hc_1$. We will rely heavily on the
apparatus of Meijer G-functions introduced and studied by the Dutch
mathematician C. S. Meijer~\cite{Meijer1946} and defined as
Mellin-Barnes integrals of the form~\cite{Adamchik90}:
\begin{align}
G^{m,n}_{p,q}\left( x \left|
\begin{matrix} \ba_p \\
\bb_q
\end{matrix} \; \right.  \right) = \frac{1}{2\pi i} \int_{C} g(\ba_p, \bb_q, \eta) x^{-\eta}
d\eta,
\end{align}
where
\begin{align}
g(\ba_p, \bb_q, \eta) =  \frac{\prod_{j=1}^m(b_j+\eta)
\prod_{j=1}^n(1-a_j-\eta)}{ \prod_{j=n+1}^p(a_j+\eta)
\prod_{j=m+1}^q(1-b_j-\eta)}.
\end{align}
For the convenience of the reader in Appendix~\ref{app:MeijerG} we provide some key
identities and the G-function related notation that will be further used in the proofs.

\subsection{Fused test statistic under $\Hc_0$}
\label{ssec:Test_statistic_under_H0}

Under $\mathcal{H}_0$ we have that the components of the fused test
statistic: $S_{\bx}$, $R_{\bx}$ and $S_{\by}$, $R_{\by}$, defined
in~\eqref{eqn:individual_test_statistics}, are central chi-square
distributed random variates with $c_{\bx} = N-\rank(\bP_{\bx})$,
$d_{\bx}=\rank(\bP_{\bx})$ and $c_{\by} = M-\rank(\bP_{\by})$,
$d_{\by}=\rank(\bP_{\by})$ degrees of freedom
respectively~\cite{Kay1998}. In this section we are interested in
the $\Hc_0$ distribution of the derived test statistic represented
as the random variable $Z = Z_{\bx}^{\frac{N}{2}}
Z_{\by}^{\frac{M}{2}}$.

\begin{thm}
The PDF and the CDF of the random variable $Z$ under hypothesis $\Hc_0$ have the following expressions, respectively:
\begin{align} \label{eqn:pdf_H0}
p_{Z|\Hc_0}(z|\Hc_0) &= \frac{2 \Gamma(\frac{c_{\bx}+d_{\bx}}{2})
\Gamma(\frac{c_{\by}+d_{\by}}{2})}{N^{\frac{d_{\bx}}{2}}
M^{\frac{d_{\by}}{2}} \Gamma(\frac{c_{\bx}}{2})
\Gamma(\frac{c_{\by}}{2})} z^{\frac{2}{N}-1} \nonumber\\
&\times G^{0,M+N}_{M+N,M+N}\left( z^2 \left|
\begin{matrix} \Delta(N,-\frac{c_{\bx}}{2}), \Delta(M,1-\frac{c_{\by}}{2}-\frac{M}{N}) \\
\Delta(N,-\frac{c_{\bx}}{2}-\frac{d_{\bx}}{2}),
\Delta(M,1-\frac{c_{\by}}{2}-\frac{d_{\by}}{2}-\frac{M}{N})
\end{matrix} \; \right.  \right).
\end{align}
\begin{align} \label{eqn:probability_of_false_alarm}
P_{Z|\Hc_0}(z|\Hc_0) &= \frac{\Gamma(\frac{c_{\bx}+d_{\bx}}{2})
\Gamma(\frac{c_{\by}+d_{\by}}{2})}{N^{\frac{d_{\bx}}{2}}
M^{\frac{d_{\by}}{2}} \Gamma(\frac{c_{\bx}}{2})
\Gamma(\frac{c_{\by}}{2})} z^{\frac{2}{N}} G^{0,M+N+1}_{M+N+1,M+N+1} \nonumber\\
&\times \left( z^2 \left|
\begin{matrix} \Delta(N,-\frac{c_{\bx}}{2}), \Delta(M,1-\frac{c_{\by}}{2}-\frac{M}{N}), 1-\frac{1}{N} \\
\Delta(N,-\frac{c_{\bx}}{2}-\frac{d_{\bx}}{2}),
\Delta(M,1-\frac{c_{\by}}{2}-\frac{d_{\by}}{2}-\frac{M}{N}),
-\frac{1}{N}
\end{matrix} \; \right.  \right).
\end{align}
\end{thm} \label{thm:1}

\begin{proof}
The outline of the proof that appears in Appendix~\ref{app:Proof_1} is as
follows: 1) represent $Z_{\by}$ via joint distribution of
$R_{\by}+S_{\by}$ and $S_{\by}$ in terms of H-function of two
variables 2) find the distribution of $Z_{\by}^{-1}$ using Theorem
4.1, case IV from Kellogg and Barnes \cite{Kellogg1987} 3) apply
random variable transformation to calculate the distribution of
$Z_{\by}^{M/N}$, 4) repeat these steps for $Z_{\bx}$, 5) find the
distribution of $Z_{\bx} Z_{\by}^{M/N}$ via multiplicative
convolution, 6) find the distribution of $(Z_{\bx}
Z_{\by}^{M/N})^{N/2} = Z_{\bx}^{N/2} Z_{\by}^{M/2}$ via Jacobian
method for random variable transformations.
\end{proof}

Note that $1-P_{Z|\Hc_0}(z|\Hc_0)$ provides us with the expression
for the probability of false alarm for the fused test
statistic~\eqref{eq:L_G_final}.

\subsection{Fused test statistic under $\Hc_1$}
\label{ssec:Test_statistic_under_H1}

Under $\mathcal{H}_1$ we have that the components of the fused test
statistic $S_{\bx}$ and $S_{\by}$ are central chi-square distributed
random variates with $c_{\bx} = N-\rank(\bP_{\bx})$ and $c_{\by} =
M-\rank(\bP_{\by})$ degrees of freedom respectively. The components
$R_{\bx}$, $R_{\by}$ are non-central chi-square variates with
degrees of freedom $d_{\bx}=\rank(\bP_{\bx})$,
$d_{\by}=\rank(\bP_{\by})$ and non-centrality parameters
$\lambda_{\bx} = \btheta_{\bx}^T \bH_{\bx}^T \bR_{\bxi}^{-1}
\bH_{\bx} \btheta_{\bx}$, $\lambda_{\by} = \btheta_{\by}^T
\bH_{\by}^T \bR_{\bzeta}^{-1} \bH_{\by} \btheta_{\by}$ respectively~\cite{Kay1998}.

In this section we are interested in the $\Hc_1$ distribution of the
fused test statistic represented, as before, as the random variable $Z =
Z_{\bx}^{\frac{N}{2}} Z_{\by}^{\frac{M}{2}}$. Under $\Hc_1$,
$Z_{\bx}$ and $Z_{\by}$ contain non-centrally distributed components
and approach we used in section~\ref{ssec:Test_statistic_under_H0}
does not seem to be applicable. We thus exploit a different technique
to identify the fused test statistic distribution under $\Hc_1$. This technique is summarized in the outline of the proof of Theorem~\ref{theorem:2}.

\begin{thm} \label{theorem:2}
The PDF and the CDF of the random variable $Z$ under hypothesis $\Hc_1$ have the following expressions, respectively:
\begin{align} \label{eqn:pdf_H1}
p_{Z |\Hc_1}(z|\Hc_1) &= \frac{2 C_{\bx} C_{\by}}{N
z^{\frac{c_{\bx}}{N}+1}} \sum_{k=0}^\infty
\frac{z^{-\frac{2k}{N}}}{k!} G^{2,1}_{3,4}\left(
\frac{\lambda_{\bx}}{2} \left|
\begin{matrix} \ba_{\bx}, 0 \\
k, \bb_{\bx} \end{matrix} \; \right.  \right) G^{1,2}_{3,4}\left(
\frac{\lambda_{\by}}{2} \left|
\begin{matrix} 0, \ba_{\by} \\
\bb_{\by}, k^\prime
\end{matrix} \; \right.  \right) \Gamma(-k^\prime) \nonumber\\
&+ \frac{2 C_{\bx} C_{\by}}{M z^{\frac{c_{\by}}{M}+1}}
\sum_{m=0}^\infty \frac{z^{-\frac{2m}{M}}}{m!} G^{2,1}_{3,4}\left(
\frac{\lambda_{\by}}{2} \left|
\begin{matrix} \ba_{\by}, 0 \\
m, \bb_{\by} \end{matrix} \; \right.  \right)  G^{1,2}_{3,4}\left(
\frac{\lambda_{\bx}}{2} \left|
\begin{matrix} 0, \ba_{\bx} \\
\bb_{\bx}, m^\prime
\end{matrix} \; \right.  \right) \Gamma(-m^\prime).
\end{align}
\begin{align} \label{eqn:probability_of_detection_fused}
P_{Z |\Hc_1}(z|\Hc_1) &= 1
\nonumber\\
&- \frac{C_{\bx} C_{\by}}{z^{\frac{c_{\bx}}{N}}} \sum_{k=0}^\infty
\frac{z^{-\frac{2k}{N}}}{k!} G^{2,1}_{3,4}\left(
\frac{\lambda_{\bx}}{2} \left|
\begin{matrix} \ba_{\bx}, 0 \\
k, \bb_{\bx} \end{matrix} \; \right.  \right) G^{1,2}_{3,4}\left(
\frac{\lambda_{\by}}{2} \left|
\begin{matrix} 0, \ba_{\by} \\
\bb_{\by}, k^\prime
\end{matrix} \; \right.  \right) \frac{\Gamma(-k^\prime)}{\frac{c_{\bx}}{2}+k} \nonumber\\
&- \frac{C_{\bx} C_{\by}}{z^{\frac{c_{\by}}{M}}} \sum_{m=0}^\infty
\frac{z^{-\frac{2m}{M}}}{m!} G^{2,1}_{3,4}\left(
\frac{\lambda_{\by}}{2} \left|
\begin{matrix} \ba_{\by}, 0 \\
m, \bb_{\by} \end{matrix} \; \right.  \right)  G^{1,2}_{3,4}\left(
\frac{\lambda_{\bx}}{2} \left|
\begin{matrix} 0, \ba_{\bx} \\
\bb_{\bx}, m^\prime
\end{matrix} \; \right.  \right) \frac{\Gamma(-m^\prime)}{\frac{c_{\by}}{2}+m}.
\end{align}
With indices $m^\prime = m\frac{N}{M} +\frac{c_{\by}N}{2M}-\frac{c_{\bx}}{2}$, $k^\prime = k\frac{M}{N} +\frac{c_{\bx}M}{2N}-\frac{c_{\by}}{2}$; constants $C_{\by} = \frac{\pi 2^{\frac{d_{\by}}{2}-1}e^{-\frac{\lambda_{\by}}{2}}}{\Gamma(\frac{c_{\by}}{2}) \lambda_{\by}^{^{\frac{d_{\by}}{2}-1}}}$, $C_{\bx} = \frac{\pi 2^{\frac{d_{\bx}}{2}-1}e^{-\frac{\lambda_{\bx}}{2}}}{\Gamma(\frac{c_{\bx}}{2}) \lambda_{\bx}^{^{\frac{d_{\bx}}{2}-1}}}$ and coefficient vectors $\ba_{\bx} = [-\frac{c_{\bx}}{2},  \frac{d_{\bx}-1}{2}]$, $\bb_{\bx} =
[\frac{d_{\bx}}{2}-1, 0, \frac{d_{\bx}-1}{2}]$ and $\ba_{\by} = [-\frac{c_{\by}}{2},  \frac{d_{\by}-1}{2}]$, $\bb_{\by} =
[\frac{d_{\by}}{2}-1, 0, \frac{d_{\by}-1}{2}]$.
\end{thm}

\begin{proof}
The outline of the proof that appears in Appendix~\ref{app:Proof_2} is as
follows: 1) construct the joint distribution of $U = R_{\by}+S_{\by}$ and $W =
S_{\by}$, 2) find moment generating function $\Mc_{\frac{W}{U}}(s)$
of random variable $W/U$, 3) find the PDF of
$Z_{\by}^{-1} = W/U$ using inverse Laplace transform of
$\Mc_{\frac{W}{U}}$, 4) apply random variable
transformation to calculate the distribution of $Z_{\by}^{-M/2}$, 5)
repeat these steps for $Z_{\bx}^{-N/2}$, 6) find the distribution of
$Z_{\bx}^{N/2} Z_{\by}^{M/2}$ via multiplicative convolution and
reciprocal transformation.
\end{proof}

Note that $1-P_{Z |\Hc_1}(z|\Hc_1)$ gives us the probability of
detection for the fused decision rule.

\section{Discussion of Results}
\label{sec:Discussion_of_Results}

In the previous section we have derived the expressions for the
probability of false alarm and the probability of detection for the
fused GLR based decision rule developed in
Section~\ref{sec:GLRT_based_fusion_rule}. The expression for the
probability of false alarm (see
eq.~\eqref{eqn:probability_of_false_alarm}) can be used to set the
detection threshold for the test~\eqref{eq:Tst} using Neyman-Pearson
criterion. The expression for the probability of detection (see
eq.~\eqref{eqn:probability_of_detection_fused}) can be used to
analyze the performance of the fused detection rule.

Note that the probability of false alarm of a single sensor can be
calculated using
expression~\eqref{eqn:probability_of_false_alarm_single} for sensor
$\bx$ (and similar one for sensor $\by$). Using technique developed
in Section~\ref{ssec:Test_statistic_under_H1} we can also derive the
expression for the $\Hc_1$ CDF for a single sensor:
\begin{align} \label{eqn:probability_of_detection_single}
P_{Z_{\bx} |\Hc_1}(z|\Hc_1) & = 1 -
\frac{C_{\bx}}{z^{\frac{c_{\bx}}{2}}} \sum_{k=0}^\infty
\frac{z^{-k}}{k!} \frac{1}{\frac{c_{\bx}}{2}+k} G^{2,1}_{3,4}\left(
\frac{\lambda_{\bx}}{2} \left|
\begin{matrix} \ba_{\bx}, 0 \\
k, \bb_{\bx} \end{matrix} \; \right.  \right).
\end{align}

\begin{figure*}[t]  
\centering \includegraphics[width =
7.5cm]{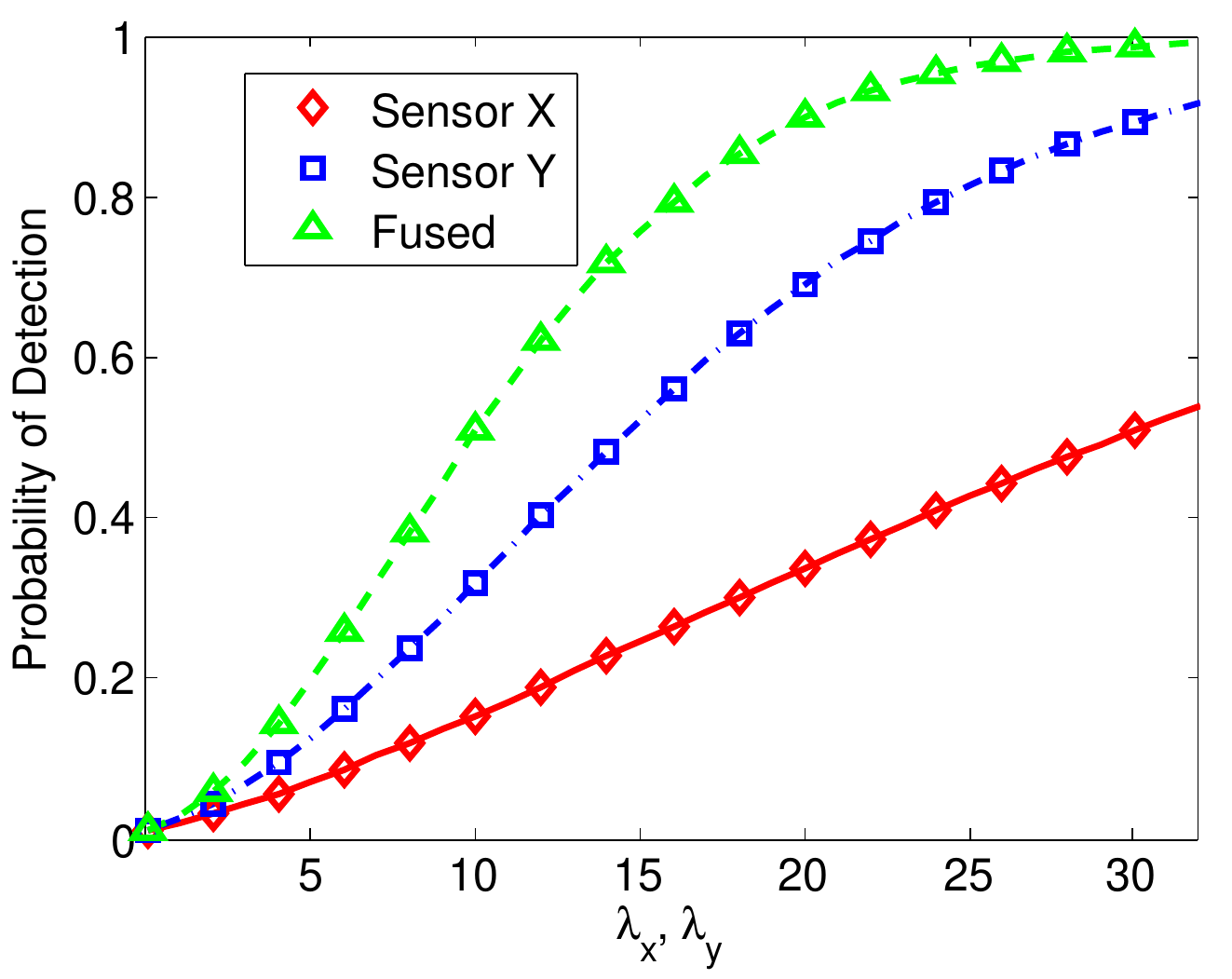} \caption{Probability of
detection for the probability of false alarm fixed at 0.01.}
\label{fig:detection_characteristics}
\end{figure*}

Next we provide the following illustrative performance analysis
example. We fix the number of samples for sensor outputs, $N=6$ and
$M=16$. We fix the signal subspace degrees of freedom $c_{\bx} =
2$ and $c_{\by} = 3$ (sensor $\bx$ has two unknown parameters and
sensor $\by$ has three unknown parameters). We vary the
noncentrality parameters $\lambda_{\bx}$ and $\lambda_{\by}$ in the
range $(0; 30]$. The resulting detection probability curves for the
probability of false alarm fixed at the level $0.01$ obtained using
equations~\eqref{eqn:probability_of_false_alarm},
\eqref{eqn:probability_of_detection_fused},
\eqref{eqn:probability_of_false_alarm_single}
and~\eqref{eqn:probability_of_detection_single} are shown in
Fig.~\ref{fig:detection_characteristics}.

Figure~\ref{fig:detection_characteristics} demonstrates that fusion
provides significant gain in terms of detection reliability even in
the case when one of the fused sources has significantly better
detection characteristics than the other. In other words, it seems
that adding even a relatively weak detector to the fusion rule may
result in significant improvement in joint detection performance.

\section{Concluding Remarks}
\label{section::Concluding Remarks}

In this paper we developed the fusion rule for joint detection with
parametric signal uncertainty and noise nuisance parameters. We
considered classical linear observation model that includes many
practical detection problems as special cases. In our model we also
incorporated uncertainty regarding noise variance. Within this
framework we derived the fusion rule based on the generalized
likelihood ratio paradigm and obtained the expressions
characterizing probability of false alarm and probability of
detection for the derived fusion rule. Analytical expressions
developed in this paper provide important research tools. From the
theoretical standpoint, they form a basis for analytical
manipulation and general study of fused distributions. From the
practical point of view, our expressions provide guidelines for
building fusion architecture and tools for direct numerical
evaluation of detection performance of this architecture in a
situation with concrete fixed parameters of the individual sensors
constituting the joint detection system. In the future we would like
to extend our current results by considering the joint GLRT
detection problem with completely unknown covariance matrices.

\appendix

\section{Useful formulae}\label{app:useful_formulae}

\subsection{Expressions for the ML estimators} \label{app:MLEs}

The MLEs of the unknown parameters can be shown to be:
\begin{align}
    \hat{\theta}_{\bx} &= \left( \bH_{\bx}^T \bR_{\bxi}^{-1} \bH_{\bx} \right)^{-1} \bH_{\bx}^T \bR_{\bxi}^{-1} \bx;~
    \hat{\theta}_{\by} = \left( \bH_{\by}^T \bR_{\bzeta}^{-1} \bH_{\by} \right)^{-1} \bH_{\by}^T \bR_{\bzeta}^{-1} \by ; \\
    \hat{\sigma}^2_{\bxi,1} &= \frac{1}{N} \bx^T \bR_{\bxi}^{-1/2} \mathbf{P}_{\bx}^{\bot} \bR_{\bxi}^{-1/2} \bx ;~
    \hat{\sigma}^2_{\bxi,0} = \frac{1}{N}\bx^T \bR_{\bxi}^{-1} \bx ;  \\
    \hat{\sigma}^2_{\bzeta,1} &= \frac{1}{M} \by^T \bR_{\bzeta}^{-1/2} \mathbf{P}_{\by}^{\bot} \bR_{\bzeta}^{-1/2} \by ;~
    \hat{\sigma}^2_{\bzeta,0} = \frac{1}{M}\by^T \bR_{\bzeta}^{-1} \by ;
\end{align}
In the expressions above $\mathbf{P}_{\bx}^{\bot} = \bI -
\mathbf{P}_{\bx}$, $\mathbf{P}_{\by}^{\bot} = \bI -
\mathbf{P}_{\by}$ and the signal projection matrices
$\mathbf{P}_{\bx}$ and $\mathbf{P}_{\by}$ are given by
$\mathbf{P}_{\bx} = \bR_{\bxi}^{-1/2} \bH_{\bx} \left( \bH_{\bx}^T
\bR_{\bxi}^{-1} \bH_{\bx} \right)^{-1} \bH_{\bx}^T
\bR_{\bxi}^{-1/2}$ and $\mathbf{P}_{\by} = \bR_{\bzeta}^{-1/2}
\bH_{\by} \left( \bH_{\by}^T \bR_{\bzeta}^{-1} \bH_{\by}
\right)^{-1} \bH_{\by}^T \bR_{\bzeta}^{-1/2}$ respectively. It is
straightforward to verify that $\mathbf{P}_{\bx}^{\bot}
\mathbf{P}_{\bx}^{\bot} = \mathbf{P}_{\bx}^{\bot}$,
$\mathbf{P}_{\by}^{\bot} \mathbf{P}_{\by}^{\bot} =
\mathbf{P}_{\by}^{\bot}$; $\mathbf{P}_{\bx} \mathbf{P}_{\bx} =
\mathbf{P}_{\bx}$, $\mathbf{P}_{\by} \mathbf{P}_{\by} =
\mathbf{P}_{\by}$; and $\mathbf{P}_{\bx}^{\bot} \mathbf{P}_{\bx} =
\bZ$, $\mathbf{P}_{\by}^{\bot} \mathbf{P}_{\by} = \bZ$.
\subsection{Meijer-G function identities} \label{app:MeijerG}

These and many other identities can be readily found in~\cite{Adamchik90,Prudnikov2003,Erdelyi1953}.
\begin{align}
x^t G^{m,n}_{p,q}\left( x \left|
\begin{matrix} \ba_p \\
\bb_q
\end{matrix} \; \right.  \right) = G^{m,n}_{p,q}\left( x \left|
\begin{matrix} \ba_p+t \\
\bb_q+t
\end{matrix} \; \right.  \right),
\end{align}
\begin{align} \label{eqn:G-function_inversion_identity}
G^{m,n}_{p,q}\left( \frac{1}{x} \left|
\begin{matrix} \ba_p \\
\bb_q
\end{matrix} \; \right.  \right) = G^{n,m}_{q,p}\left( x \left|
\begin{matrix} 1-\bb_p \\
1-\ba_q
\end{matrix} \; \right.  \right),
\end{align}
\begin{align} \label{eqn:G-function_definite_integration_identity}
\int_{0}^\infty &x^{\alpha-1} G^{s,t}_{u,v}\left( \sigma x \left|
\begin{matrix} \bc_u \\
\bd_v
\end{matrix} \; \right.  \right) G^{m,n}_{p,q}\left( \omega x^{\frac{\ell}{k}} \left|
\begin{matrix} \ba_p \\
\bb_q
\end{matrix} \; \right.  \right) d x = \frac{k^\mu
\ell^{\rho+\alpha(v-u)-1}\sigma^{-\alpha}}{(2\pi)^{b^\star(\ell-1)+c^\star(k-1)}}
\nonumber\\
&\times G^{km+\ell t,kn+\ell s}_{kp+\ell v,kq+\ell u}\left(
\frac{\omega^k k^{k(p-q)}}{\sigma^\ell \ell^{\ell(u-v)}} \left|
\begin{matrix} \be_{kp+\ell v}  \\
\bbf_{kq+\ell u}
\end{matrix} \; \right.  \right),
\end{align}
where $b^\star = s+t-(u+v)/2$, $\rho = \sum_{j=1}^v d_j -
\sum_{j=1}^u c_j +(u+v)/2+1$, $c^\star = m+n-(p+q)/2$, $\rho =
\sum_{j=1}^q b_j - \sum_{j=1}^p a_j +(p+q)/2+1$, $\be_{kp+\ell v} =
[\Delta(k,a_1), \ldots, \Delta(k,a_n), \Delta(\ell, 1-\alpha-d_1),
\ldots, \Delta(\ell, 1-\alpha-d_v), \Delta(k,a_{n+1}), \ldots,
\Delta(k,a_p)]$, $\bbf_{kq+\ell u} = [\Delta(k,b_1), \ldots,
\Delta(k,b_m), \Delta(\ell, 1-\alpha-c_1), \ldots, \Delta(\ell,
1-\alpha-c_u), \Delta(k,b_{m+1}), \ldots, \Delta(k,b_q)]$. Here we
have utilized the following notation: $\Delta(k,a_j) =
\frac{a_j}{k}, \frac{a_j+1}{k}, \ldots, \frac{a_j+k-1}{k}$.

We close the list of useful G-function formulae with the indefinite
integration expression:
\begin{align} \label{eqn:G-function_indefinite_integration_identity}
\int_{0}^y &x^{\alpha-1} G^{m,n}_{p,q} \left( \omega x \left|
\begin{matrix} \ba_p \\
\bb_q
\end{matrix} \; \right.  \right) d x = y^\alpha G^{m,n+1}_{p+1,q+1}\left(
\omega y \left|
\begin{matrix} a_1,\ldots,a_n, 1-\alpha,a_{n+1},\ldots,a_{p}  \\
b_1,\ldots,b_m, -\alpha,b_{m+1},\ldots,b_{q}
\end{matrix} \; \right.  \right).
\end{align}

\section{Proof of Theorem 4.1} \label{app:Proof_1}

First denote $U = R_{\by} + S_{\by}$ and $W = S_{\by}$. Using the
Jacobian method for random variable transformation one can show that
\begin{equation}
p_{W, U}(w, u) = p_{X}(u-w) p_{Y}(w).
\end{equation}
Taking into account the fact that $p_{R_{\by}}(x) = \frac{1}{\Gamma(\frac{d_{\by}}{2})
2^{\frac{d_{\by}}{2}-1}} x^{\frac{d_{\by}}{2}-1} e^{-\frac{x}{2}}$
and $p_{S_{\by}}(x) = \frac{1}{\Gamma(\frac{c_{\by}}{2})
2^{\frac{c_{\by}}{2}-1}} y^{\frac{c_{\by}}{2}-1} e^{-\frac{x}{2}}$
are central chi-square distributions with degrees of freedom
$c_{\by}$, $d_{\by}$ and substituting these into the previous
expression we obtain:
\begin{align}
p_{W, U}(w, u) = \frac{(u-w)^{\frac{d_{\by}}{2}-1}
w^{\frac{c_{\by}}{2}-1} e^{-\frac{u}{2}}}{\Gamma(\frac{d_{\by}}{2})
\Gamma(\frac{c_{\by}}{2}) 2^{\frac{c_{\by}}{2} +
\frac{d_{\by}}{2}}}, \quad 0 < w < u.
\end{align}

This expression exactly corresponds to the McKay's bivariate gamma
distribution (Kellogg and Barnes, \cite[p. 213]{Kellogg1987}) if we
set the parameters of this distribution $a=1/2$,
$p=\frac{c_{\by}}{2}$, $q=\frac{d_{\by}}{2}$ (here we refer to the
Kellogg and Barnes' original notation). It thus can be represented
as the bivariate H-function (\cite[p. 213]{Kellogg1987})
\begin{align}
p_{W, U}(w, u) = \frac{(1/2)^2}{\Gamma(\frac{c_{\by}}{2})}
H^{1,0,0,0,1,0}_{1,1,0,0,0,1}\left(
\begin{matrix} \frac{1}{ 2}w \\
\\
\frac{1}{ 2} u
\end{matrix} \left|
\begin{matrix} (\frac{c_{\by}}{2}+\frac{d_{\by}}{2}-2, 1) \\
(\frac{c_{\by}}{2}+\frac{d_{\by}}{2}-1, 1), - \\
-\\
(\frac{c_{\by}}{2}-1, 1), -
\end{matrix} \; \right.  \right).
\end{align}
We can now find the distribution of random variable $V = W U^{-1}$ using Theorem 4.1, case IV (Kellogg
and Barnes \cite[p. 213]{Kellogg1987}):
\begin{align}
p_{V}(v) =
\frac{\Gamma(\frac{c_{\by}}{2}+\frac{d_{\by}}{2})}{\Gamma(\frac{c_{\by}}{2})}
H^{1,0}_{1,1}\left( v \left|
\begin{matrix} (\frac{c_{\by}}{2}+\frac{d_{\by}}{2}-1, 1) \\
(\frac{c_{\by}}{2}-1, 1)
\end{matrix} \; \right.  \right).
\end{align}
Using the relationship between the H-function and the
G-function~\cite[p. 531]{Prudnikov2003} we can further simplify this expression:
\begin{align}
p_{V}(v) =
\frac{\Gamma(\frac{c_{\by}}{2}+\frac{d_{\by}}{2})}{\Gamma(\frac{c_{\by}}{2})}
G^{1,0}_{1,1}\left( v \left|
\begin{matrix} \frac{c_{\by}}{2}+\frac{d_{\by}}{2}-1 \\
\frac{c_{\by}}{2}-1
\end{matrix} \; \right.  \right).
\end{align}
The last expression gives the pdf of
$\frac{S_{\by}}{R_{\by}+S_{\by}}$. To find the pdf of
$Z_{\by}^{M/N}$ we use the fact that $Z_{\by}^{M/N} =
\frac{1}{V^{M/N}}$ and apply the Jacobian transformation method:
\begin{align}
p_{Z_{\by}^{M/N}}(z) = \left. \frac{p_{V}(v)}{|\frac{\partial
z^{M/N}}{\partial v}|} \right|_{v = \frac{1}{z^{N/M}}} = \frac{N}{M}
\frac{p_{V}(\frac{1}{z^{N/M}})}{(z^{N/M})^{M/N+1}}
\end{align}
This results in the following expression:
\begin{align}
p_{Z_{\by}^{M/N}}(z)  &= \frac{N}{M}
\frac{\Gamma(\frac{c_{\by}}{2}+\frac{d_{\by}}{2})}{\Gamma(\frac{c_{\by}}{2})}
\frac{G^{1,0}_{1,1}\left( z^{-\frac{N}{M}} \left|
\begin{matrix} \frac{c_{\by}}{2}+\frac{d_{\by}}{2}-1 \\
\frac{c_{\by}}{2}-1
\end{matrix} \; \right.  \right)}{(z^{N/M})^{M/N+1}} \\
&= \frac{N}{M}
\frac{\Gamma(\frac{c_{\by}}{2}+\frac{d_{\by}}{2})}{\Gamma(\frac{c_{\by}}{2})}
G^{0,1}_{1,1}\left( z^{\frac{N}{M}} \left|
\begin{matrix} 1-\frac{c_{\by}}{2}-\frac{M}{N} \\
1-\frac{c_{\by}}{2}-\frac{d_{\by}}{2}-\frac{M}{N}
\end{matrix} \; \right.  \right).
\end{align}
Similarly, the distribution of $Z_{\bx}$ appears to be:
\begin{align} \label{eqn:probability_of_false_alarm_single}
p_{Z_{\bx}}(z)  =
\frac{\Gamma(\frac{c_{\bx}}{2}+\frac{d_{\bx}}{2})}{\Gamma(\frac{c_{\bx}}{2})}
G^{0,1}_{1,1}\left( z \left|
\begin{matrix} -\frac{c_{\bx}}{2} \\
-\frac{c_{\bx}}{2}-\frac{d_{\bx}}{2}
\end{matrix} \; \right.  \right).
\end{align}
The next step is to find the PDF of $Z = Z_{\bx} Z_{\by}^{M/N}$.
Using the Jacobian technique again one can show that
\begin{align} \label{eqn:PDF_of_product}
p_{Z_{\bx}Z_{\by}^{M/N}}(z) = \int_{0}^z
\frac{p_{Z_{\bx}}(\frac{z}{v}) p_{Z_{\by}^{M/N}}(v)}{v} d v
\end{align}
Substituting the PDFs $p_{Z_{\bx}}$ and $p_{Z_{\by}^{M/N}}$ obtained
in the previous steps, utilizing the fact that $p_{Z_{\bx}}(z/v) = 0
\text{ for } v>z$ and using the G-function identity~\eqref{eqn:G-function_inversion_identity} and
integration formula~\eqref{eqn:G-function_definite_integration_identity} we obtain:
\begin{align}
p_{Z_{\bx}Z_{\by}^{M/N}}(z) &=
\frac{N\Gamma(\frac{c_{\bx}}{2}+\frac{d_{\bx}}{2})
\Gamma(\frac{c_{\by}}{2}+\frac{d_{\by}}{2})}{M\Gamma(\frac{c_{\bx}}{2})
\Gamma(\frac{c_{\by}}{2})} \int_{0}^\infty G^{0,1}_{1,1}\left(
\frac{v}{z} \left|
\begin{matrix} 1+\frac{c_{\bx}}{2}+\frac{d_{\bx}}{2} \\
1+\frac{c_{\bx}}{2}
\end{matrix} \; \right.  \right) \nonumber\\
&\times G^{0,1}_{1,1}\left( v^{\frac{N}{M}} \left|
\begin{matrix} 1-\frac{c_{\by}}{2}-\frac{M}{N} \\
1-\frac{c_{\by}}{2}-\frac{d_{\by}}{2}-\frac{M}{N}
\end{matrix} \; \right.  \right) \frac{d v}{v} \nonumber\\
&= \frac{\Gamma(\frac{c_{\bx}+d_{\bx}}{2})
\Gamma(\frac{c_{\by}+d_{\by}}{2})}{N^{\frac{d_{\bx}}{2}-1}
M^{\frac{d_{\by}}{2}} \Gamma(\frac{c_{\bx}}{2})
\Gamma(\frac{c_{\by}}{2})} \nonumber\\
&\times G^{0,M+N}_{M+N,M+N}\left( z^N \left|
\begin{matrix} \Delta(N,-\frac{c_{\bx}}{2}), \Delta(M,1-\frac{c_{\by}}{2}-\frac{M}{N}) \nonumber \\
\Delta(N,-\frac{c_{\bx}}{2}-\frac{d_{\bx}}{2}),
\Delta(M,1-\frac{c_{\by}}{2}-\frac{d_{\by}}{2}-\frac{M}{N})
\end{matrix} \; \right.  \right).
\end{align}
The final step of the proof is to find the PDF of $Z =
Z_{\bx}^{N/2} Z_{\by}^{M/2}$ by applying the transformation
$(\cdot)^{^{N/2}}$ to the random variable $Z_{\bx} Z_{\by}^{M/N}$.
After some algebra, this results in expression~\eqref{eqn:pdf_H0}.

Finally, we find the corresponding CDF by using the chage of
variable $y = x^2$ in the integral below and by applying the
indefinite G-function integration formula~\eqref{eqn:G-function_indefinite_integration_identity}, resulting in~\eqref{eqn:probability_of_false_alarm}.

\section{Proof of Theorem 4.2} \label{app:Proof_2}

First denote, as previously, $U = R_{\by} + S_{\by}$ and $W =
S_{\by}$. Now recall that the PDFs of $R_{\by}$ and $S_{\by}$ under
$\Hc_1$ can be written as follows:
\begin{align}
p_{R_{\by}|\Hc_1}(x|\Hc_1) &= \frac{1}{2}
e^{-\frac{x+\lambda_{\by}}{2}}
\left(\frac{x}{\lambda_{\by}}\right)^{\frac{d_{\by}}{4}-\frac{1}{2}}
\BesselI_{\frac{d_{\by}}{2}-1}(\sqrt{\lambda_{\by} x}) \nonumber\\
p_{S_{\by}|\Hc_1}(x|\Hc_1) &=
\frac{x^{\frac{c_{\by}}{2}-1}}{2^{\frac{c_{\by}}{2}}\Gamma(\frac{c_{\by}}{2})}
e^{-\frac{x}{2}}.
\end{align}
Here $\BesselI_{\frac{d_{\by}}{2}-1}(\cdot)$ is the modified Bessel
function of the first kind. Using the relationship between this
function and the Meijer G-function~\cite{Wolfram2011} we can write the
$\Hc_1$-hypothesis joint distribution of $U$ and $W$ as follows:
\begin{align}
p_{U,W}(u,w) &= p_{R_{\by}|\Hc_1}(u-w)p_{S_{\by}|\Hc_1}(w) , \quad 0 \leq w < u < \infty \nonumber\\
&= C_1 e^{-\frac{u}{2}} w^{\frac{c_{\by}}{2}-1} G^{1,0}_{1,3}\left(
\frac{\lambda_{\by}(u-w)}{4} \left|
\begin{matrix} \frac{d_{\by}-1}{2} \\
\frac{d_{\by}}{2}-1, 0, \frac{d_{\by}-1}{2}
\end{matrix} \; \right.  \right).
\end{align}
Where $C_1 = \frac{\pi
2^{\frac{d_{\by}}{2}-\frac{c_{\by}}{2}-2}}{e^{\frac{\lambda_{\by}}{2}}
\Gamma(\frac{c_{\by}}{2}) \lambda_{\by}^{^{\frac{d_{\by}}{2}-1}}}$.
Using the definition of Meijer G-function we can write the following
integral representation of the G-function above:
\begin{align}
G^{1,0}_{1,3}\left( \frac{\lambda_{\by}(u-w)}{4} \left|
\begin{matrix} \frac{d_{\by}-1}{2} \\
\frac{d_{\by}}{2}-1, 0, \frac{d_{\by}-1}{2}
\end{matrix} \; \right.  \right) = \frac{1}{2\pi i} \int_{C} g_1(\eta) \left(\frac{\lambda_{\by}}{4}\right)^{-\eta} (u-w)^{-\eta} d\eta.
\end{align}
Since $0 \leq w < u$ this integral can be expanded in the uniformly
convergent series:
\begin{align} \label{eqn:integral_series}
\int_{C} g_1(\eta) \left(\frac{\lambda_{\by}}{4}\right)^{-\eta}
(u-w)^{-\eta} d\eta = \sum_{k=0}^\infty \int_{C} \frac{u^{-k-\eta}
w^{k}}{k!} g_1(\eta) \frac{\Gamma(\eta+k)}{\Gamma(\eta)}
\left(\frac{\lambda_{\by}}{4}\right)^{-\eta} d\eta.
\end{align}
Where the order of integration and summation can be interchanged
because of the uniform convergence.

We can now write down the expression for the moment generating
function of random variable $\frac{W}{U}$
\begin{align}
\Mc_{\frac{W}{U}}(s) = \int_0^{\infty} \int_0^{u} e^{-\frac{w}{u} s}
p_{W,U}(w,u) du \, dw.
\end{align}
The application of~\eqref{eqn:integral_series} and some
reorganization of the above formula lead to
\begin{align} \label{eqn:first_expression_for_mGF}
\Mc_{\frac{W}{U}}(s) = \frac{C_1}{2\pi i} \sum_{k=0}^\infty \int_{C}
\int_{0}^\infty e^{-\frac{u}{2}} \frac{u^{-k-\eta}}{k!} \Jc_k(u,s) d
u\, g_1(\eta) \frac{\Gamma(\eta+k)}{\Gamma(\eta)}
\left(\frac{\lambda_{\by}}{4}\right)^{-\eta} d\eta.
\end{align}
Here we have
\begin{align}
\Jc_k(u,s) &= \int_0^u e^{-\frac{w}{u} s} w^{k+\frac{c_{\by}}{2}-1}
d w \nonumber\\
&=\left[\Gamma(k+\frac{c_{\by}}{2})-\Gamma(k+\frac{c_{\by}}{2}, s)
\right] \left(\frac{s}{u}\right)^{-k-\frac{c_{\by}}{2}}.
\end{align}
Furthermore, since
\begin{align}
\int_0^\infty e^{-\frac{u}{2}}
\left(\frac{s}{u}\right)^{-k-\frac{c_{\by}}{2}} u^{-k-\eta} d u =
s^{-k-\frac{c_{\by}}{2}}
\Gamma(1-\eta+\frac{c_{\by}}{2})2^{1-\eta+\frac{c_{\by}}{2}},
\end{align}
we can denote $C_2 = 2^{1+\frac{c_{\by}}{2}} C_1$ and
simplify~\eqref{eqn:first_expression_for_mGF} as follows:
\begin{align} \label{eqn:simplified_expression_for_mGF}
\Mc_{\frac{W}{U}}(s) &= \frac{C_2}{2\pi i} \sum_{k=0}^\infty
\frac{s^{-k-\frac{c_{\by}}{2}}}{k!}
\left[\Gamma(k+\frac{c_{\by}}{2})-\Gamma(k+\frac{c_{\by}}{2}, s)
\right] \nonumber\\
&\times \int_{C} g_1(\eta) \Gamma(1-\eta+\frac{c_{\by}}{2})
\frac{\Gamma(\eta+k)}{\Gamma(\eta)}
\left(\frac{\lambda_{\by}}{2}\right)^{-\eta} d\eta.
\end{align}

Next we can write down the expression for PDF of random variable
$W/U$ using the fact that it is equal to the inverse Laplace
transform $\Lc^{-1}$ of $\Mc_{\frac{W}{U}}$ and noting that for
$0\leq z\leq1$ we have
\begin{align}
\Lc^{-1}\left\{s^{-k-\frac{c_{\by}}{2}}
\left[\Gamma(k+\frac{c_{\by}}{2})-\Gamma(k+\frac{c_{\by}}{2}, s)
\right]\right\} = z^{k+\frac{c_{\by}}{2}-1},
\end{align}
that leads to the following expression:
\begin{align}
p_{\frac{W}{U}}(z) = \frac{C_2}{2\pi i} \int_{C}
\underbrace{\sum_{k=0}^\infty \frac{z^{k+\frac{c_{\by}}{2}-1}}{k!}
\frac{\Gamma(\eta+k)}{\Gamma(\eta)}}_{z^{\frac{c_{\by}}{2}-1}(1-z)^{-\eta}}
g_1(\eta) \Gamma(1-\eta+\frac{c_{\by}}{2})
\left(\frac{\lambda_{\by}}{2}\right)^{-\eta} d\eta.
\end{align}
Transforming this back to the G-function domain, recalling that
$Z_{\by} = \frac{U}{W}, 1\leq Z_{\by} \leq \infty$ and denoting
$C_{\by} = \frac{\pi
2^{\frac{d_{\by}}{2}-1}e^{-\frac{\lambda_{\by}}{2}}}{\Gamma(\frac{c_{\by}}{2})
\lambda_{\by}^{^{\frac{d_{\by}}{2}-1}}}$ results in:
\begin{align}
p_{Z_{\by}^{-1}}(z) = C_{\by} z^{\frac{c_{\by}}{2}-1}
G^{1,1}_{2,3}\left( \frac{\lambda_{\by}}{2}(1-z) \left|
\begin{matrix} -\frac{c_{\by}}{2},  \frac{d_{\by}-1}{2} \\
\frac{d_{\by}}{2}-1, 0, \frac{d_{\by}-1}{2}
\end{matrix} \; \right.  \right).
\end{align}
By analogy, we have for $Z_{\bx}^{-1}$:
\begin{align}
p_{Z_{\bx}^{-1}}(z) = C_{\bx} z^{\frac{c_{\bx}}{2}-1}
G^{1,1}_{2,3}\left( \frac{\lambda_{\bx}}{2}(1-z) \left|
\begin{matrix} -\frac{c_{\bx}}{2},  \frac{d_{\bx}-1}{2} \\
\frac{d_{\bx}}{2}-1, 0, \frac{d_{\bx}-1}{2}
\end{matrix} \; \right.  \right),
\end{align}
where $C_{\bx} = \frac{\pi
2^{\frac{d_{\bx}}{2}-1}e^{-\frac{\lambda_{\bx}}{2}}}{\Gamma(\frac{c_{\bx}}{2})
\lambda_{\bx}^{^{\frac{d_{\bx}}{2}-1}}}$.

Using the Jacobian method for random variable transformation one can
further show that the PDFs for $Z_{\bx}^{-\frac{N}{2}}$ and
$Z_{\by}^{-\frac{M}{2}}$ take on the form:
\begin{align}
p_{Z_{\bx}^{-\frac{N}{2}}}(z) &= \frac{2 C_{\bx}}{N}
z^{\frac{c_{\bx}}{N}-1} G^{1,1}_{2,3}\left(
\frac{\lambda_{\bx}}{2}(1-z^{\frac{2}{N}}) \left|
\begin{matrix} -\frac{c_{\bx}}{2},  \frac{d_{\bx}-1}{2} \\
\frac{d_{\bx}}{2}-1, 0, \frac{d_{\bx}-1}{2}
\end{matrix} \; \right.  \right), \\
p_{Z_{\by}^{-\frac{M}{2}}}(z) &= \frac{2 C_{\by}}{M}
z^{\frac{c_{\by}}{M}-1} G^{1,1}_{2,3}\left(
\frac{\lambda_{\by}}{2}(1-z^{\frac{2}{M}}) \left|
\begin{matrix} -\frac{c_{\by}}{2},  \frac{d_{\by}-1}{2} \\
\frac{d_{\by}}{2}-1, 0, \frac{d_{\by}-1}{2}
\end{matrix} \; \right.  \right).
\end{align}

We can now find the PDF of random variable $V =
Z_{\bx}^{-\frac{N}{2}} Z_{\by}^{-\frac{M}{2}}$ using formula
analogous to~\eqref{eqn:PDF_of_product} and using the fact that
$0\leq Z_{\bx}^{-\frac{N}{2}} \leq 1$ and $0\leq
Z_{\by}^{-\frac{M}{2}} \leq 1$ results in $V \leq
Z_{\by}^{-\frac{M}{2}} \leq 1$. Upon denoting $\ba_{\bx} =
[-\frac{c_{\bx}}{2},  \frac{d_{\bx}-1}{2}]$, $\bb_{\bx} =
[\frac{d_{\bx}}{2}-1, 0, \frac{d_{\bx}-1}{2}]$ and $\ba_{\by} =
[-\frac{c_{\by}}{2},  \frac{d_{\by}-1}{2}]$, $\bb_{\by} =
[\frac{d_{\by}}{2}-1, 0, \frac{d_{\by}-1}{2}]$ we have
\begin{align}
p_{Z_{\bx}^{-\frac{N}{2}} Z_{\by}^{-\frac{M}{2}}}(z) &= \frac{4
C_{\bx} C_{\by}}{M N} \int_{z}^{1} \frac{u^{\frac{c_{\by}}{M}-1}}{u}
\left(\frac{z}{u}\right)^{\frac{c_{\bx}}{N}-1} G^{1,1}_{2,3}\left(
\frac{\lambda_{\bx}}{2}(1-\left(\frac{z}{u}\right)^{\frac{2}{N}})
\left|
\begin{matrix} \ba_{\bx} \\
\bb_{\bx} \end{matrix} \; \right.  \right) \nonumber\\
&\times G^{1,1}_{2,3}\left(
\frac{\lambda_{\by}}{2}(1-u^{\frac{2}{M}}) \left|
\begin{matrix} \ba_{\by} \\
\bb_{\by}
\end{matrix} \; \right.  \right) d u.
\end{align}
Using the integral representation of the Meijer G-function we can
further write it as
\begin{align}
p_{Z_{\bx}^{-\frac{N}{2}} Z_{\by}^{-\frac{M}{2}}}(z) &= \frac{4
C_{\bx} C_{\by}}{M N} z^{\frac{c_{\bx}}{N}-1} \int_C\int_L
g(\eta,\ba_{\bx}, \bb_{\bx}) g(\omega,\ba_{\by}, \bb_{\by})
\nonumber\\
&\times \left(\frac{\lambda_{\bx}}{2}\right)^{-\eta}
\left(\frac{\lambda_{\by}}{2}\right)^{-\omega} \Ic(z, \eta,
\omega)d\eta\,d\omega
\end{align}
Where the integrand in $\Ic(z, \eta, \omega)$:
\begin{align}
\Ic(z, \eta, \omega) = \int_{z}^{1}
u^{\frac{c_{\by}}{M}-\frac{c_{\bx}}{N}-1}
\left[1-\left(\frac{z}{u}\right)^{\frac{2}{N}}\right]^{-\eta}
\left[1-u^{\frac{2}{M}}\right]^{-\omega} d u.
\end{align}
can be expanded into the double uniformly convergent series since $z
\leq u\leq 1$ and $0\leq z\leq 1$. After changing the order of
integration and summation (valid due to the uniform convergence) and
evaluating the integral this yields:
\begin{align} \label{eqn:intermediate_result_1}
\Ic(z, \eta, \omega) &= \sum_{k=0}^\infty\sum_{m=0}^\infty
\frac{1}{k!m!}
\frac{\Gamma(\eta+k)\Gamma(\omega+m)}{\Gamma(\eta)\Gamma(\omega)}
\int_{z}^{1} u^{\frac{c_{\by}}{M}-\frac{c_{\bx}}{N}-1}
\left(\frac{z}{u}\right)^{\frac{2k}{N}} u^{\frac{2 m}{M}} d u
\nonumber\\
&=  \sum_{k=0}^\infty\sum_{m=0}^\infty
\frac{\Gamma(\eta+k)\Gamma(\omega+m)}{k!m!
\Gamma(\eta)\Gamma(\omega)}
\frac{\frac{MN}{2}(z^{\frac{2m}{M}}z^{\frac{c_{\by}}{M}-\frac{c_{\bx}}{N}}
 -z^{\frac{2k}{N}})}{\frac{M}{2}(2k + c_{\bx}) - \frac{N}{2}(2m +
c_{\by})}.
\end{align}
Changing the order of summation, substituting the following two
expressions
\begin{align}
\sum_{m=0}^\infty \frac{\Gamma(\omega+m)}{m! \Gamma(\omega)}
\frac{\frac{MN}{2}}{\frac{M}{2}(2k + c_{\bx}) - \frac{N}{2}(2m +
c_{\by})} &= \frac{M}{2}
\frac{\Gamma(1-\omega)\Gamma(\frac{c_{\by}}{2}-k\frac{M}{N}
-\frac{c_{\bx}M}{2N})}{\Gamma(1-\omega+\frac{c_{\by}}{2}-k\frac{M}{N}
-\frac{c_{\bx}M}{2N})} \nonumber\\
\sum_{k=0}^\infty \frac{\Gamma(\eta+k)}{k! \Gamma(\eta)}
\frac{\frac{MN}{2}}{\frac{M}{2}(2k + c_{\bx}) - \frac{N}{2}(2m +
c_{\by})} &= -\frac{N}{2}
\frac{\Gamma(1-\eta)\Gamma(\frac{c_{\bx}}{2}-m\frac{N}{M}
-\frac{c_{\by}N}{2M})}{\Gamma(1-\eta+\frac{c_{\bx}}{2}-m\frac{N}{M}
-\frac{c_{\by}N}{2M})}.
\end{align}
into~\eqref{eqn:intermediate_result_1} and returning to the
G-function representation of the Mellin-Barnes integrals we have the
following expression for the PDF of random variable
$Z_{\by}^{-\frac{N}{2}} Z_{\bx}^{-\frac{M}{2}}$:
\begin{align}
p_{Z_{\bx}^{-\frac{N}{2}} Z_{\by}^{-\frac{M}{2}}}(z) &= \frac{2
C_{\bx} C_{\by}}{N} z^{\frac{c_{\bx}}{N}-1}\sum_{k=0}^\infty
\frac{z^{\frac{2k}{N}}}{k!} G^{2,1}_{3,4}\left(
\frac{\lambda_{\bx}}{2} \left|
\begin{matrix} \ba_{\bx}, 0 \\
k, \bb_{\bx} \end{matrix} \; \right.  \right) \nonumber\\
&\times G^{1,2}_{3,4}\left( \frac{\lambda_{\by}}{2} \left|
\begin{matrix} 0, \ba_{\by} \\
\bb_{\by}, k\frac{M}{N} +\frac{c_{\bx}M}{2N}-\frac{c_{\by}}{2}
\end{matrix} \; \right.  \right) \Gamma(\frac{c_{\by}}{2}-k\frac{M}{N}
-\frac{c_{\bx}M}{2N}) \nonumber\\
&+ \frac{2 C_{\bx} C_{\by}}{M}
z^{\frac{c_{\by}}{M}-1}\sum_{m=0}^\infty \frac{z^{\frac{2m}{M}}}{m!}
G^{2,1}_{3,4}\left( \frac{\lambda_{\by}}{2} \left|
\begin{matrix} \ba_{\by}, 0 \\
m, \bb_{\by} \end{matrix} \; \right.  \right) \nonumber\\
&\times G^{1,2}_{3,4}\left( \frac{\lambda_{\bx}}{2} \left|
\begin{matrix} 0, \ba_{\bx} \\
\bb_{\bx}, m\frac{N}{M} +\frac{c_{\by}N}{2M}-\frac{c_{\bx}}{2}
\end{matrix} \; \right.  \right) \Gamma(\frac{c_{\bx}}{2}-m\frac{N}{M}-\frac{c_{\by}N}{2M}).
\end{align}
Finally, using the reciprocal transformation of the random variable
$Z_{\by}^{-\frac{N}{2}} Z_{\bx}^{-\frac{M}{2}}$ and denoting
$m^\prime = m\frac{N}{M} +\frac{c_{\by}N}{2M}-\frac{c_{\bx}}{2}$,
$k^\prime = k\frac{M}{N} +\frac{c_{\bx}M}{2N}-\frac{c_{\by}}{2}$ we
obtain the desired PDF $p_{Z_{\bx}^{\frac{N}{2}}
Z_{\by}^{\frac{M}{2}}}(z)$ in expression~\eqref{eqn:pdf_H1}.

The CDF follows straightforwardly via the term-wise integration of~\eqref{eqn:pdf_H1}
using the fact that $Z\geq1$, which implies
\begin{align}
P_{Z |\Hc_1}(z|\Hc_1) &= \int_1^z p_{Z |\Hc_1}(u|\Hc_1) du.
\end{align}
and evaluation of the last integral leads to~\eqref{eqn:probability_of_detection_fused}.

\bibliographystyle{imsart-nameyear}
\bibliography{bibliography}

%
%
%
%
%
%

\end{document}